\documentclass[12pt, reqno, twoside, letterpaper]{amsart}

\usepackage{z-Paperstyle}
\def\fq{{\mathfrak q}}
\def\fX{{\mathfrak X}}

\def\EOS{\mathsf{EOS}}
\def\QRH{\mathsf{QRH}}

\title[The EOS conjecture for prime-power moduli]{On the
Erd\H{o}s--Odlyzko--S\'ark\"ozy conjecture\\ for prime-power moduli}

\author{William Banks}
\address{Department of Mathematics\\
University of Missouri\\
Columbia, MO 65211\\
USA}
\email{bankswd@missouri.edu}

\author{Igor E. Shparlinski}
\address{School of Mathematics and Statistics\\
University of New South Wales\\
Sydney, NSW 2052\\
Australia}
\email{igor.shparlinski@unsw.edu.au}


\date{\today}

\begin{document}

\dedicatory{Dedicated to the memory of Nikolai Mikhailovich Korobov} 

\begin{abstract}
Erd\H{o}s, Odlyzko and S\'ark\"ozy conjectured that every reduced residue class modulo $q$ contains a product of two primes not exceeding $q$. O. Klurman, A. P. Mangerel and J. Ter\"av\"ainen (2023) have shown that three primes suffice when $q$ is sufficiently smooth. We study the moduli $q=p^\gamma$ that are powers of a fixed prime $p$ and we show that the third prime can be taken subpolynomially small. For any  $q_\bullet\le q$ for which $q_\bullet\ge\exp\bigl(C(\log q)^{2/3}(\log\log q)^{4/3}\bigr)$, we give an asymptotic formula for the weighted number of representations of each reduced class as a product of two primes not exceeding $q$ and a third not exceeding~$q_\bullet$. We also show that the conjecture holds for all but a proportion $\exp\bigl(-c(\log q)^{1/3}(\log\log q)^{-1/3}\bigr)$ of the reduced classes. Here $C$ and $c$ are positive constants that depend only on $p$. Both results are unconditional and they rely on a wide zero-free region with no exceptional zeros for the Dirichlet $L$-functions attached to characters modulo $p^\gamma$.We also present conditional versions of these results for arbitrary $q$ under a quasi-Riemann hypothesis.
\end{abstract}

\maketitle

\tableofcontents


\section{Introduction and statement of results}
\label{sec:intro}

For integers $k,q\ge 1$ and a real number $x\ge 1$, we consider the
set of reduced residue classes in $(\Z/q\Z)^\times$ that contain a product of
$k$ primes not exceeding $x$; that is, we write
\[
\EOS_k(q;x)\defeq\bigl\{a\in(\Z/q\Z)^\times:
\ell_1\cdots\ell_k\equiv a\bmod q
\text{~for some primes~}\ell_1,\ldots,\ell_k\le x\bigr\},
\]
and we set
\[
\EOS_k(q) = \EOS_k(q;q).
\]

For the binary case, Erd\H{o}s, Odlyzko and S\'ark\"ozy~\cite{EOS}
have conjectured that $\EOS_2(q)=(\Z/q\Z)^\times$ for all
large $q$, that is, every class $a\in(\Z/q\Z)^\times$
contains a product $\ell_1\ell_2$ of two primes $\ell_1,\ell_2\le q$.
For general moduli $q$, the strongest unconditional results are due
to Matom\"aki and Ter\"av\"ainen~\cite{MT}, who showed that
$|\EOS_2(q)|\ge(\tfrac23-\eps)\varphi(q)$ for every $\eps>0$ and all
large~$q$, and that $\EOS_3(q)=(\Z/q\Z)^\times$ for all large
cube-free $q$. The barrier they have identified is a possible real
character $\psi\bmod q$ with $\psi(\ell)=-1$ for almost all primes
$\ell\le q$; this is precisely the behavior that a Siegel zero would
produce.

In this paper we study the EOS conjecture for
prime-power moduli $q=p^\gamma$, where $p$ is a fixed prime
and $\gamma$ is large. It has been known since the work
of Postnikov~\cite{Po55,Po56} that character sums for
such moduli satisfy bounds that are significantly stronger than
what is available in general. Combining Postnikov's idea with the
method of Korobov~\cite{Ko69,Ko74} for double Weyl sums, the
authors~\cite{BS19,BS20} have shown that the Dirichlet
\mbox{$L$-functions} attached to characters modulo $q=p^\gamma$ have
no Siegel zeros and that their low-lying zeros are excluded from a
region that is far larger than the classical one.
For such moduli the ternary problem is already settled. Klurman,
Mangerel and Ter\"av\"ainen~\cite{KMT} showed that
$\EOS_3(q)=(\Z/q\Z)^\times$ for all large $q$ that are
$q^{\eps}$-smooth, once $\eps>0$ is small enough.
Our first result sharpens this for $q=p^\gamma$ by allowing the
third prime to be much smaller than $q$ even on the logarithmic scale.

\begin{theorem}[Ternary case, prime-power moduli]\label{thm:ternary}
For any fixed prime $p$, there are constants $\gamma_1$, $C_1>0$ and $c_1>0$ 
such that the following holds. Let $q=p^\gamma$ with
$\gamma\ge\gamma_1$ and suppose that a real $q_\bullet\le q$ satisfies
\be\label{eq:deranged}
\log q_\bullet\ge C_1(\log q)^{2/3}(\log\log q)^{4/3}.
\ee  
Then for every $a\in(\Z/q\Z)^\times$ we have
\[
\ssum{\ell_1,\ell_2\le q,\,\ell_3\le q_\bullet\\
\ell_1\ell_2\ell_3\equiv a\bmod q}
\log\ell_1\cdot\log\ell_2\cdot\log\ell_3
=\frac{q^2q_\bullet}{\varphi(q)}\biggl\{1+O\biggl(\exp\Bigl(
-\frac{c_1\log q_\bullet}{(\log q)^{2/3}(\log\log q)^{1/3}}
\Bigr)\biggr)\biggr\},
\]
the sum being over primes and the implied constant depending only
on $p$. 
\end{theorem}

Applying Theorem~\ref{thm:ternary} with equality
in~\eqref{eq:deranged}, we see that if $\gamma$ is large enough
then every $a\in(\Z/q\Z)^\times$ is the product
$\ell_1\ell_2\ell_3$ of three primes with $\ell_1,\ell_2\le q$ and
\[
\ell_3\le\exp\bigl(C_1(\log q)^{2/3}(\log\log q)^{4/3}\bigr)
=q^{o(1)}\qquad(\gamma\to\infty).
\]

\begin{theorem}[Binary case, prime-power moduli]\label{thm:binary}
For any fixed prime $p$, there are constants $\gamma_2$ and $C_2>0$
such that for every modulus $q=p^\gamma$ with $\gamma\ge\gamma_2$
we have
\[
\bigl|(\Z/q\Z)^\times\setminus \EOS_2(q)\bigr|
\le\varphi(q)\exp(-C_2(\log q)^{1/3}(\log\log q)^{-1/3}).
\]
\end{theorem}

We now turn to an \emph{arbitrary} modulus $q\ge 3$. Our remaining
results are conditional on the following quasi-Riemann hypothesis
for the $L$-functions attached to characters modulo $q$.

\medskip\noindent\textsc{Hypothesis $\QRH[q,\sigma_0]$.}
{\itshape For every nonprincipal character $\chi$ modulo $q$, the
function $L(s,\chi)$ does not vanish in the half-plane
$\sigma>\sigma_0$.}
\smallskip

\begin{theorem}[Ternary case, general moduli]
\label{thm:conditional-3} 
There is a constant $C>0$, depending only on $\sigma_0$, such that
under $\QRH[q,\sigma_0]$ every $a\in(\Z/q\Z)^\times$ is the product
$\ell_1\ell_2\ell_3$ of three primes with $\ell_1,\ell_2\le q$ and
$\ell_3\le C(\log q)^{2/(1-\sigma_0)}$.
\end{theorem}

Theorem~\ref{thm:conditional-3} is derived from an asymptotic
formula similar to that of Theorem~\ref{thm:ternary};
see~\eqref{eq:cond ternary}. In that formula the third prime
carries the weight $1-\ell_3/q_\bullet$ and for this reason we do
not state it here. Using the proof
of~\cite[Equation~(13.20)]{MV} in place of
Lemma~\ref{lem:conditional} below, one obtains a formula with the
same weights as in Theorem~\ref{thm:ternary} at the cost of an
additional factor $\log q_\bullet$ in the error term.

\begin{theorem}[Binary case, general moduli] \label{thm:conditional-2}
Under $\QRH[q,\sigma_0]$ we have
\[
\bigl|(\Z/q\Z)^\times\setminus\EOS_2(q)\bigr|
\ll\varphi(q)\,q^{-2(1-\sigma_0)}(\log q)^{3}.
\]
\end{theorem}

Under $\QRH[q,\tfrac12]$, which is
the \emph{generalized Riemann hypothesis} for all nonprincipal
characters modulo $q$, Theorem~\ref{thm:conditional-3} shows that
the third prime can be chosen with $\ell_3\ll(\log q)^4$ and
Theorem~\ref{thm:conditional-2} shows that all but $O((\log q)^3)$
reduced classes lie in $\EOS_2(q)$.

Our approach can be adapted to the more general sets
$\EOS_k(q;x)$, but apart from the range of the third prime in
Theorem~\ref{thm:ternary}, we confine ourselves to the classical
case $x=q$.

Throughout the paper, the letter $\ell$ always denotes a prime. In
our unconditional results and their proofs, $p$ is a fixed prime,
$q=p^\gamma$ with $\gamma\in\N$, and all constants, whether named
or implied by the symbols $O$, $\ll$ and $\gg$, depend at most on
$p$. In our conditional results and their proofs, $q\ge 3$ is an
arbitrary modulus, $\sigma_0\in[\tfrac12,1)$ is fixed and all
constants depend at most on $\sigma_0$.

\section{Character sums over primes}
\label{sec:charsums}

\subsection{Preparations}
In this subsection, $q$ is an arbitrary modulus. 
For every Dirichlet character $\chi\bmod q$ we put
\[
\psi(x,\chi)\defeq\sum_{n\le x}\Lambda(n)\chi(n)
\mand
\vartheta(x,\chi)\defeq\sum_{\ell\le x}\chi(\ell)\log\ell,
\]
and we write $\psi(x)$ and $\vartheta(x)$ for the ``untwisted''
Chebyshev functions. Following Riemann, we use the notation
$s=\sigma+it$ with $\sigma,t\in\R$.

We use the prime number theorem in the classical form
\be\label{eq:PNT}
\vartheta(x)=x+O\bigl(x\exp(-c\sqrt{\log x}\,)\bigr)
\qquad(x\ge 2),
\ee
where $c>0$ is an absolute constant;
see~\cite[Theorem~6.9]{MV}. Integrating, we also have
\be\label{eq:PNT weighted}
\sum_{\ell\le x}\Bigl(1-\frac{\ell}{x}\Bigr)\log\ell
=\frac1x\int_1^x\vartheta(t)\,dt
=\frac{x}{2}+O\bigl(x\exp(-c\sqrt{\log x}\,)\bigr)
\ee
for all $x\ge 2$ provided that $c$ is small enough, as we
assume from now on. We remark that using the stronger
Korobov--Vinogradov form of the prime number
theorem (see, e.g.,~\cite[Corollary~8.30]{IwKow}) in place
of~\eqref{eq:PNT} would not improve any of our main results; only
the first two error terms in~\eqref{eq:cond ternary} would change.

We also record a bound on the second moment of
$\vartheta(q,\chi)$. The primes $\ell\le q$ with  $\ell\nmid q$ lie
in distinct classes of $(\Z/q\Z)^\times$, so the orthogonality of
characters and the Chebyshev bound $\vartheta(q)\ll q$ give
\be\label{eq:second moment}
\sum_{\chi\bmod q}\bigl|\vartheta(q,\chi)\bigr|^2
=\varphi(q)\ssum{\ell\le q\\ \ell\,\nmid\,q}(\log\ell)^2
\le \varphi(q)\vartheta(q)\log q 
\ll\varphi(q)q\log q.
\ee

\subsection{Bounds for prime-power moduli}
In this subsection, we quote several results from the literature that
are proved only for \emph{primitive} characters, whereas our
aim is to treat every \emph{nonprincipal} character $\chi\bmod q$.
For a prime-power modulus the passage from one to the other is immediate.
Indeed, the primitive character $\fX$ that induces $\chi$ has
conductor $\fq=p^{\gamma_\star}$ for some $1\le\gamma_\star\le\gamma$,
and as $q$ and $\fq$ are powers of the same prime $p$,
it follows that $\chi$ and $\fX$ are the \emph{same function} and we have
\[
L(s,\chi)=L(s,\fX),\qquad
\psi(x,\chi)=\psi(x,\fX),\qquad
\vartheta(x,\chi)=\vartheta(x,\fX).
\]
Accordingly, any result known for primitive characters applies to 
$\chi$ with the conductor $\fq$ in place of $q$. 
The upper bounds that we quote are increasing in the modulus,
so they remain valid with $q$ in place of $\fq$. One exception is
the zero-free region of~\cite{BS20}, which requires $\fq$ to be large, a restriction
that is removed in Lemma~\ref{lem:zerofree} below.

Our principal tool is a zero-free region for the functions $L(s,\chi)$
with $\chi\bmod q$, in which no exceptional zero can occur.  For primitive
characters of large conductor the region is supplied by~\cite[Theorem~12]{BS20};
for characters with a smaller conductor we appeal to an
earlier result of Iwaniec~\cite[Theorem~2]{Iw74}.

\begin{lemma}\label{lem:zerofree}
For each prime $p$, there are constants $\gamma_0\ge 1$ and
$A_0\in(0,1]$ such that for every modulus $q=p^\gamma$
with $\gamma\ge\gamma_0$ and every nonprincipal character
$\chi\bmod q$, the function $L(s,\chi)$ does not vanish in the
region
\[
\euR_q\defeq\bigl\{s\in\C:\sigma>1-\eta_q,\ |t|\le T_q\bigr\},
\]
where
\[
\eta_q\defeq\frac{A_0}{(\log q)^{2/3}(\log\log q)^{1/3}}
\mand
T_q\defeq\exp\bigl((\log q)^{7/8}\bigr).
\]
\end{lemma}

\begin{proof}
To begin, we recall the zero-free region of~\cite[Theorem~12]{BS20} in the
case of a prime-power modulus. There are constants
$\gamma_\sharp\ge 1$, $A>0$ and $B>0$ 
such that for every primitive character $\fX$ modulo
$\fq=p^{\gamma_\star}$ with $\gamma_\star\ge\gamma_\sharp$, the
function $L(s,\fX)$ does not vanish in the region defined by
\be\label{eq:BS region}
\sigma>1-\frac{A}{(\log\fq)^{2/3}(\log\log\fq)^{1/3}}
\mand
|t|\le\exp\bigl\{B(\log\fq)^{5/3}(\log\log\fq)^{1/3}\bigr\}.
\ee

Now let $\chi$ be a nonprincipal character modulo $q=p^\gamma$
with $\gamma\ge\gamma_0$, where $\gamma_0$ is sufficiently large in
terms of $p$. Let $\fX$ be the primitive character modulo
$\fq=p^{\gamma_\star}$ that induces $\chi$. Since
$L(s,\chi)=L(s,\fX)$, we must show that $L(s,\fX)$ does not vanish
in $\euR_q$.

First suppose that $\log\fq>(\log q)^{7/8}$. By our choice of
$\gamma_0$, this inequality implies that
$\gamma_\star\ge\gamma_\sharp$ and that
\[
B(\log\fq)^{5/3}(\log\log\fq)^{1/3}\ge(\log q)^{7/8};
\]
hence the second condition in~\eqref{eq:BS region} holds whenever
$|t|\le T_q$. Since $\fq\le q$, the first condition holds whenever
$\sigma>1-\eta_q$ provided that $A_0\le A$. With these choices
every point of $\euR_q$ satisfies~\eqref{eq:BS region} and the
lemma is proved for every character $\chi$ whose conductor
satisfies $\log\fq>(\log q)^{7/8}$.

For the remaining characters we recall the zero-free region of
Iwaniec~\cite[Theorem~2]{Iw74}. Let $\fX$ be a primitive character
modulo $\fq$, let $d$ be the product of the distinct primes
dividing $\fq$ and put $l\defeq\log\fq(|t|+3)$. Then $L(s,\fX)$
does not vanish when
\be\label{eq:Iw region}
\sigma>1-\frac{1}{4\cdot 10^4\{\log d+(l\log 2l)^{3/4}\}},
\ee
except possibly at a real zero in the case that $\fX$ is a real
character.

Now suppose that $\log\fq\le(\log q)^{7/8}$. In our situation
$d=p$, and for $|t|\le T_q$ it is easy to see that
$l\le 3(\log q)^{7/8}$, therefore
\[
4\cdot 10^4\{\log p+(l\log 2l)^{3/4}\}
\le(\log q)^{2/3}\le\eta_q^{-1}
\]
if $\gamma_0$ is large enough and $A_0\le 1$. Thus every point of
$\euR_q$ satisfies~\eqref{eq:Iw region} and so the only zeros of
$L(s,\fX)$ that can lie in $\euR_q$ are real zeros in the case
that $\fX$ is real.

It remains to exclude such zeros. A real primitive character whose
conductor is a power of $p$ has conductor $p$ if $p$ is odd and
conductor $4$ or $8$ if $p=2$; thus there are at most three such
characters and they are determined by $p$. Since $L(1,\fX)\ne 0$
and the zeros of $L(s,\fX)$ are isolated, there is a number
$\beta_p<1$, depending only on $p$, such that none of the
corresponding functions $L(\sigma,\fX)$ vanishes for
$\sigma>\beta_p$. As $\eta_q\le A_0$, the lemma follows with
the choice $A_0\defeq\min\{1,A,1-\beta_p\}$.
\end{proof}

To bound $\psi(x,\chi)$, we use the truncated explicit formula for
primitive characters together with the partial fraction expansion
of $L'/L$ in the forms given by Montgomery and
Vaughan~\cite[Theorem~12.10 and Lemma~12.6]{MV}. As usual, we put
\[
\psi_0(x,\chi)\defeq\tfrac12\bigl\{\psi(x^-,\chi)+\psi(x^+,\chi)\bigr\}
\qquad(x>1),
\]
the mean of the left and right limits of $\psi(y,\chi)$ as $y\to x$,
we write $\langle x\rangle$ for the distance from $x$ to the
nearest prime power other than $x$ itself and we denote
$\tau\defeq|t|+4$.

\begin{lemma}\label{lem:MV12.10}
Let $\fX$ be a primitive character modulo $\fq\ne 1$. For all $x,T\ge 2$
we have uniformly
\[
\psi_0(x,\fX)=-\ssum{\rho=\beta+i\gamma\\ |\gamma|\le T}
\frac{x^\rho}{\rho}
-\tfrac12\log(x-1)
-\tfrac12\fX(-1)\log(x+1)
+C(\fX)+R(x,T;\fX),
\]
where $\rho$ runs over the nontrivial zeros of $L(s,\fX)$,
\[
C(\fX)\defeq\frac{L'}{L}(1,\fX)+\log\frac{\fq}{2\pi}-C_0
\]
with $C_0=0.5772\cdots $ the Euler--Mascheroni constant, and
\[
R(x,T;\fX)\ll(\log x)\min\Bigl\{1,\frac{x}{T\langle x\rangle}\Bigr\}
+\frac{x}{T}(\log \fq Tx)^2.
\] 
\end{lemma}

\begin{lemma}\label{lem:MV12.6}
Let $\fX$ be a primitive character modulo $\fq\ne 1$. The estimate
\[
\frac{L'}{L}(s,\fX)=\ssum{\rho=\beta+i\gamma\\ |\gamma-t|\le 1}
\frac1{s-\rho}+O(\log \fq\tau)
\]
holds uniformly for $-1\le\sigma\le 2$, where $\rho$ runs over the
nontrivial zeros of $L(s,\fX)$.
\end{lemma}

The main result of this  subsection is the following, in which
$\eta_q$ and $T_q$ are as in Lemma~\ref{lem:zerofree}.

\begin{proposition}\label{prop:pointwise}
For each prime $p$, there is a constant $\gamma_0$
such that for every modulus $q=p^\gamma$ with $\gamma\ge\gamma_0$
and every nonprincipal character $\chi\bmod q$, the bound 
\be\label{eq:pointwise}
\max\bigl\{\bigl|\psi(x,\chi)\bigr|,\bigl|\vartheta(x,\chi)\bigr|\bigr\}
\ll x(\log qx)^2\bigl(x^{-\eta_q}+T_q^{-1}\bigr)
\ee
holds for all $x\ge 2$.
\end{proposition}

\begin{proof}
Let $\gamma_0$ be the constant of Lemma~\ref{lem:zerofree}, which
we assume is large enough so that $\eta_q\le\tfrac12$.
Let $\chi$ be a nonprincipal character modulo
$q=p^\gamma$ with $\gamma\ge\gamma_0$ and let $\fX$ be the
primitive character modulo $\fq$ that induces $\chi$. Recall that
$\chi=\fX$ and $\fq\le q$.

By Lemma~\ref{lem:zerofree} the function $L(s,\fX)$ does not
vanish in $\euR_q$. Thus if $\rho=\beta+i\gamma$ is a nontrivial
zero of $L(s,\fX)$ with $|\gamma|\le T_q$, then
$\beta\le 1-\eta_q$. Since $\fX$ is primitive,
$1-\overline\rho=1-\beta+i\gamma$ is also a nontrivial zero and
therefore
\[
\eta_q\le\beta\le 1-\eta_q\qquad(|\gamma|\le T_q).
\]
We apply Lemma~\ref{lem:MV12.6} with $s=1$. Since $L(s,\fX)$ has
$O(\log q)$ zeros with $|\gamma|\le 1$
(see~\cite[Theorem~10.17]{MV}), we get that
\[
\frac{L'}{L}(1,\fX)=\ssum{\rho\\ |\gamma|\le 1}\frac1{1-\rho}
+O(\log q)\ll\eta_q^{-1}\log q.
\]
Consequently, the constant $C(\fX)$ of Lemma~\ref{lem:MV12.10}
satisfies
\be\label{eq:C bound}
C(\fX)\ll\eta_q^{-1}\log q\ll(\log q)^2.
\ee

We now prove~\eqref{eq:pointwise} for $\psi(x,\chi)$. The bound is
trivial if $x^{\eta_q}<2$, for then
\[
\bigl|\psi(x,\chi)\bigr|\le\psi(x)\ll x\le 2x^{1-\eta_q}.
\]
We therefore assume that $x^{\eta_q}\ge 2$ and we put
$T\defeq\min\{x,T_q\}$, so that $2\le T\le x$ and
\[
\frac{x}{T}\le 1+\frac{x}{T_q}\le x^{1-\eta_q}+\frac{x}{T_q}.
\]
We apply Lemma~\ref{lem:MV12.10} to $\fX$. Using~\eqref{eq:C bound} and the bounds
\[
\bigl|\psi(x,\chi)-\psi_0(x,\fX)\bigr|\le\log x
\mand
R(x,T;\fX)\ll\frac{x}{T}(\log qx)^2,
\]
we get that
\[
\psi(x,\chi)=-\ssum{\rho\\ |\gamma|\le T}
\frac{x^\rho}{\rho}
+O\bigl(x(\log qx)^2(x^{-\eta_q}+T_q^{-1})\bigr).
\]
Since $T\le T_q$, every zero in the sum has
$\eta_q\le\beta\le 1-\eta_q$, hence
\[
\ssum{\rho\\ |\gamma|\le T}\frac{x^\rho}{\rho}
\ll x^{1-\eta_q}\biggl\{\ssum{\rho\\ |\gamma|\le 1}
\frac1{\eta_q}+\ssum{\rho\\ 1<|\gamma|\le T}\frac1{|\gamma|}\biggr\}
\ll x^{1-\eta_q}\bigl\{\eta_q^{-1}\log q
+(\log qT)^2\bigr\}
\ll x^{1-\eta_q}(\log qx)^2.
\]
Combining this with the previous estimate, we obtain~\eqref{eq:pointwise} for $\psi(x,\chi)$. The same bound holds for
$\vartheta(x,\chi)$, since
\[
\bigl|\vartheta(x,\chi)-\psi(x,\chi)\bigr|
\le\ssum{\ell^k\le x\\ k\ge 2}\log\ell
\le\sum_{2\le k\le\log x/\log 2}\vartheta(x^{1/k})\ll x^{1/2}
\le x^{1-\eta_q}
\]
in view of the inequality $\eta_q\le\tfrac12$.
\end{proof}

\subsection{Conditional bounds for general moduli}
\label{sec:conditional}

In this  subsection, $q\ge 3$ is an arbitrary modulus and
$\sigma_0\in[\tfrac12,1)$ is fixed.
For every character $\chi\bmod q$ we denote
\[
\widetilde\vartheta(x,\chi)\defeq\sum_{\ell\le x}\Bigl(1-\frac{\ell}{x}\Bigr)
\chi(\ell)\log\ell.
\]
The weight $1-\ell/x$, which is taken from
Montgomery~\cite[Chapter~13]{Mo71}, makes the sum over zeros in
the explicit formula absolutely convergent. Consequently, the
bound of Lemma~\ref{lem:conditional} below has the factor $\log q$
where the corresponding bound for $\vartheta(x,\chi)$ has
 $(\log x)(\log qx)$. This saving matters for our application
since it lowers the powers of $\log q$ in
Theorems~\ref{thm:conditional-3} and~\ref{thm:conditional-2}.

\begin{lemma}\label{lem:conditional}
Under $\QRH[q,\sigma_0]$, for every nonprincipal character
$\chi\bmod q$ we have
\[
\widetilde \vartheta (x,\chi)\ll x^{\sigma_0}\log q\qquad(x\ge 2).
\]
\end{lemma}

\begin{proof}
 For $\sigma_0=\tfrac12$, this follows
from~\cite[Equation~(13.28)]{MV} on removing the prime powers.
The proof given there applies to every $\sigma_0\in[\tfrac12,1)$,
provided that $\chi$ is first replaced by the primitive character
that induces it, which changes $\widetilde\vartheta(x,\chi)$ by at
most $\log q$, and that the contour is moved to the line
$\sigma=-\tfrac12$ instead of $\sigma=\tfrac14$, since under
$\QRH[q,\sigma_0]$ the zeros need not lie to the right of the
latter line. In this way one finds that
\[
\widetilde \vartheta (x,\chi)=-\sum_\rho\frac{x^\rho}{\rho(\rho+1)}
+O\bigl(x^{1/2}+\log q\cdot\log x\bigr),
\]
where $\rho=\beta+i\gamma$ runs over the zeros of $L(s,\chi)$ in
the strip $0<\sigma<1$. Under $\QRH[q,\sigma_0]$, every such zero
satisfies $1-\sigma_0\le\beta\le\sigma_0$ by the functional
equation. Since $L(s,\chi)$ has $O(\log qn)$ zeros with
$n-1\le|\gamma|\le n$ for each $n\ge 1$
(see~\cite[Theorem~10.17]{MV}), the sum over zeros is
$O(x^{\sigma_0}\log q)$ and the result follows.
\end{proof}

Since the weights do not exceed one, the argument that
gives~\eqref{eq:second moment} also shows that
\be\label{eq:second moment weighted}
\sum_{\chi\bmod q}\bigl|\widetilde\vartheta(q,\chi)\bigr|^2
=\varphi(q)\ssum{\ell\le q\\ \ell\,\nmid\,q}
\Bigl(1-\frac{\ell}{q}\Bigr)^2(\log\ell)^2
\ll\varphi(q)q\log q.
\ee

The primes $\ell\mid q$ are the only ones missing from the sums
$\vartheta(x,\chi_0)$ and $\widetilde\vartheta(x,\chi_0)$ and
their logarithms contribute at most $\log q$ to either sum.
 Hence~\eqref{eq:PNT} and~\eqref{eq:PNT weighted} show that for
all $x\ge 2$ we have
\be\label{eq:stdPNT}
 \vartheta(x,\chi_0)=x+O\bigl(x\exp(-c\sqrt{\log x}\,)+\log q\bigr)
\ee
and
\be\label{eq:principal weighted}
\widetilde\vartheta(x,\chi_0)=\frac{x}{2}
+O\bigl(x\exp(-c\sqrt{\log x}\,)+\log q\bigr).
\ee

\section{Proofs of results for prime-power moduli}
\label{sec:proofs}

\subsection{Proof of Theorem~\ref{thm:ternary}}
Let $\gamma_0$ be as in Proposition~\ref{prop:pointwise} and let
$A_0$ be as in Lemma~\ref{lem:zerofree}. We show that the theorem
holds with $C_1\defeq 8A_0^{-1}$,  $c_1\defeq\tfrac12A_0$ and any
sufficiently large $\gamma_1\ge\gamma_0$. Let $q=p^\gamma$ with
$\gamma\ge\gamma_1$ and suppose that  $q_\bullet\le q$
satisfies~\eqref{eq:deranged}.  With our choice of $c_1$, the error
term in the theorem is $O(q_\bullet^{-\eta_q/2})$, and
by~\eqref{eq:deranged} and our choice of $C_1$ we have
\be\label{eq:log four}
q_\bullet^{\eta_q/2}=\exp\bigl(\tfrac12\eta_q\log q_\bullet\bigr)
\ge\exp\bigl(\tfrac12A_0C_1\log\log q\bigr)=(\log q)^4.
\ee
For every $a\in(\Z/q\Z)^\times$, the orthogonality of characters
gives
\begin{align*}
\Sigma_a&\defeq\ssum{\ell_1,\ell_2\le q,\,\ell_3\le q_\bullet\\
\ell_1\ell_2\ell_3\equiv a\bmod q}
\log\ell_1\cdot\log\ell_2\cdot\log\ell_3\\
&=\ssum{\ell_1,\ell_2\le q,\,\ell_3\le q_\bullet}
\log\ell_1\cdot\log\ell_2\cdot\log\ell_3\cdot\frac{1}{\varphi(q)}
\sum_{\chi\bmod q}\chi(\ell_1\ell_2\ell_3\bar a)\\
&=\frac{1}{\varphi(q)}\sum_{\chi\bmod q}\chi(\bar a)\,
\vartheta(q,\chi)^2\,\vartheta(q_\bullet,\chi),
\end{align*}
where $\bar a$ is the inverse of $a$ in $(\Z/q\Z)^\times$.

The principal character $\chi_0$ provides the main term. Since
$q_\bullet\ge p$ by~\eqref{eq:deranged}, we have
\[
\vartheta(q,\chi_0)=\vartheta(q)-\log p
\mand
\vartheta(q_\bullet,\chi_0)=\vartheta(q_\bullet)-\log p.
\]
We apply~\eqref{eq:PNT}. Since
$\log q_\bullet\le\log q$, the inequality
\[
c\sqrt{\log x}\ge\tfrac12\eta_q\log x\ge\tfrac12\eta_q\log q_\bullet
\]
holds for both $x=q_\bullet$ and $x=q$ if $\gamma_1$ is large
enough, and we also have $x^{-1}\log p\ll q_\bullet^{-\eta_q/2}$
since $\eta_q\le\tfrac12$. Consequently,
$\vartheta(x,\chi_0)=x\bigl\{1+O(q_\bullet^{-\eta_q/2})\bigr\}$ holds in
both cases and thus
\[
\frac{\vartheta(q,\chi_0)^2\vartheta(q_\bullet,\chi_0)}{\varphi(q)}
=\frac{q^2q_\bullet}{\varphi(q)}
\Bigl\{1+O\bigl(q_\bullet^{-\eta_q/2}\bigr)\Bigr\}.
\]

For the nonprincipal characters we apply~\eqref{eq:second moment}
and obtain
\begin{align*}
\frac1{\varphi(q)}\sum_{\chi\ne\chi_0}
\bigl|\vartheta(q,\chi)\bigr|^2\bigl|\vartheta(q_\bullet,\chi)\bigr|
&\le\max_{\chi\ne\chi_0}\bigl|\vartheta(q_\bullet,\chi)\bigr|
\cdot\frac1{\varphi(q)}\sum_{\chi\bmod q}\bigl|\vartheta(q,\chi)\bigr|^2\\
&\ll\frac{q^2q_\bullet}{\varphi(q)}
\cdot\frac{\log q}{q_\bullet}
\max_{\chi\ne\chi_0}\bigl|\vartheta(q_\bullet,\chi)\bigr|.
\end{align*}
We have $\log(qq_\bullet)\le 2\log q$, and
$T_q\ge q^{\eta_q}\ge q_\bullet^{\eta_q}$ if $\gamma_1$ is large
enough. Hence~\eqref{eq:pointwise} with $x=q_\bullet$ gives
\[
\max_{\chi\ne\chi_0}\bigl|\vartheta(q_\bullet,\chi)\bigr|
\ll q_\bullet^{1-\eta_q}(\log q)^2,
\]
and by~\eqref{eq:log four} we conclude that 
\[
\frac1{\varphi(q)}\sum_{\chi\ne\chi_0}
\bigl|\vartheta(q,\chi)\bigr|^2\bigl|\vartheta(q_\bullet,\chi)\bigr|
\ll\frac{q^2q_\bullet}{\varphi(q)}\,(\log q)^3q_\bullet^{-\eta_q}
\le\frac{q^2q_\bullet}{\varphi(q)}\,q_\bullet^{-\eta_q/2}.
\]
This proves the desired asymptotic formula.

\subsection{Proof of Theorem~\ref{thm:binary}}
Let $\gamma_0$ be as in Proposition~\ref{prop:pointwise} and let
$A_0$ be as in Lemma~\ref{lem:zerofree}. We show that the theorem
holds with any positive $C_2<2A_0$ and any sufficiently large
$\gamma_2\ge\gamma_0$. Let $q=p^\gamma$ with $\gamma\ge\gamma_2$.
For every $a\in(\Z/q\Z)^\times$, let
\[
f(a)\defeq\ssum{\ell_1,\ell_2\le q\\ \ell_1\ell_2\equiv a\bmod q}
\log\ell_1\cdot\log\ell_2
=\frac1{\varphi(q)}\sum_{\chi\bmod q}\overline\chi(a)\vartheta(q,\chi)^2,
\]
and note that $f$ vanishes on the set
$\cE\defeq(\Z/q\Z)^\times\setminus \EOS_2(q)$. Summing over $a\in\cE$,
we get that
\[
\sum_{\chi\bmod q}\vartheta(q,\chi)^2S(\chi)=0,
\qquad\text{where}\quad
S(\chi)\defeq\sum_{a\in\cE}\overline\chi(a).
\]
Since $S(\chi_0)=|\cE|$, it follows that
\be\label{eq:E_isolate1}
\vartheta(q,\chi_0)^2|\cE|
\le\max_{\chi\ne\chi_0}\bigl|\vartheta(q,\chi)\bigr|
\sum_{\chi\bmod q}\bigl|\vartheta(q,\chi)\bigr|\bigl|S(\chi)\bigr|.
\ee
By the orthogonality of characters we have
\be\label{eq:E_isolate2}
\sum_{\chi\bmod q}\bigl|S(\chi)\bigr|^2=\varphi(q)|\cE|.
\ee
Hence, using the Cauchy--Schwarz inequality
and~\eqref{eq:second moment}, we see that
\[
\sum_{\chi\bmod q}\bigl|\vartheta(q,\chi)\bigr|\bigl|S(\chi)\bigr|
\ll\varphi(q)(q\log q)^{1/2}|\cE|^{1/2}.
\]
Moreover, the maximum in~\eqref{eq:E_isolate1} is
$O(q^{1-\eta_q}(\log q)^2)$ by~\eqref{eq:pointwise} with
$x=q$ and the inequality $T_q\ge q^{\eta_q}$, and
$\vartheta(q,\chi_0)\gg q$ by the prime number theorem.
Combining these bounds, we obtain
\[
q^2|\cE|\ll q^{1-\eta_q}(\log q)^2
\cdot\varphi(q)(q\log q)^{1/2}|\cE|^{1/2},
\]
and thus
\[
|\cE|\ll\frac{\varphi(q)^2}{q}\,q^{-2\eta_q}(\log q)^5
\le\varphi(q)\exp\bigl(-2A_0(\log q)^{1/3}(\log\log q)^{-1/3}
+5\log\log q\bigr).
\]
The result follows from our choice of $C_2$ if $\gamma_2$ is
large enough.

\section{Proofs of conditional results for general moduli}
\label{sec:conditional-proof}

\subsection{Proof of Theorem~\ref{thm:conditional-3}}
We first show that for every  $q_\bullet\ge 2$ and every
$a\in(\Z/q\Z)^\times$ we have
\be\label{eq:cond ternary}
\begin{split}
&\ssum{\ell_1,\ell_2\le q,\,\ell_3\le q_\bullet\\
\ell_1\ell_2\ell_3\equiv a\bmod q}
\Bigl(1-\frac{\ell_3}{q_\bullet}\Bigr)
\log\ell_1\cdot\log\ell_2\cdot\log\ell_3\\
&\qquad=\frac{q^2q_\bullet}{2\varphi(q)}
\biggl\{1+O\biggl(\exp\bigl(-c\sqrt{\log q_\bullet}\,\bigr)
+\exp\bigl(-c\sqrt{\log q}\,\bigr)
+\frac{(\log q)^2}{q_\bullet^{1-\sigma_0}}\biggr)\biggr\},
\end{split}
\ee
where $c>0$ is the absolute constant in~\eqref{eq:PNT}
and~\eqref{eq:PNT weighted}.

Let $\Sigma_a$ denote the left side
of~\eqref{eq:cond ternary}. As in the proof of
Theorem~\ref{thm:ternary}, the orthogonality of characters gives
\[
\Sigma_a=\frac{1}{\varphi(q)}\sum_{\chi\bmod q}\chi(\bar a)
\vartheta(q,\chi)^2\widetilde\vartheta(q_\bullet,\chi).
\]
By~\eqref{eq:second moment} the contribution of the
nonprincipal characters to this sum is
\[
O\Bigl(q\log q\cdot\max_{\chi\ne\chi_0}
\bigl|\widetilde\vartheta(q_\bullet,\chi)\bigr|\Bigr).
\]
By Lemma~\ref{lem:conditional} we therefore have
\be\label{eq:Sigma weighted}
\Sigma_a=\frac{\vartheta(q,\chi_0)^2\widetilde\vartheta(q_\bullet,\chi_0)}
{\varphi(q)}+O\bigl(q\,q_\bullet^{\sigma_0}(\log q)^2\bigr).
\ee
 Since $q^{-1}\log q\ll\exp\bigl(-c\sqrt{\log q}\,\bigr)$, the
estimates~\eqref{eq:stdPNT} with $x=q$
and~\eqref{eq:principal weighted} with $x=q_\bullet$ give
\[
\vartheta(q,\chi_0)
=q\Bigl\{1+O\Bigl(\exp\bigl(-c\sqrt{\log q}\,\bigr)\Bigr)\Bigr\}
\]
and
\[
\widetilde\vartheta(q_\bullet,\chi_0)
=\frac{q_\bullet}{2}\Bigl\{1+O\Bigl(
\exp\bigl(-c\sqrt{\log q_\bullet}\,\bigr)
+\frac{\log q}{q_\bullet}\Bigr)\Bigr\}.
\]
Moreover,
$q_\bullet^{-1}\log q\le(\log q)^2q_\bullet^{-(1-\sigma_0)}$.
Hence the first term on the right side
of~\eqref{eq:Sigma weighted} has the form given on the right side
of~\eqref{eq:cond ternary}, and since $\varphi(q)\le q$, the
second term is
\[
O\Bigl(\frac{q^2q_\bullet}{\varphi(q)}
\cdot\frac{(\log q)^2}{q_\bullet^{1-\sigma_0}}\Bigr).
\]
This proves~\eqref{eq:cond ternary}.

Finally, let $q_\bullet\defeq C(\log q)^{2/(1-\sigma_0)}$ with
$C\ge 3$. Then  $q_\bullet\ge C$, and the  first and third error
terms in~\eqref{eq:cond ternary} satisfy
\[
\exp\bigl(-c\sqrt{\log q_\bullet}\,\bigr)
\le\exp\bigl(-c\sqrt{\log C}\,\bigr)
\mand
\frac{(\log q)^2}{q_\bullet^{1-\sigma_0}}=C^{-(1-\sigma_0)}.
\]
Hence~\eqref{eq:cond ternary} shows that $\Sigma_a>0$ for every
$a\in(\Z/q\Z)^\times$ if $C$ and $q$ are large enough in terms
of $\sigma_0$. For each of the finitely many remaining $q\ge 3$, some
prime $\ell_0\le q$ does not divide $q$ and by Dirichlet's
theorem every $a\in(\Z/q\Z)^\times$ is of the form
$\ell_0^2\ell_3$ with some prime $\ell_3$; the last assertion of the
theorem follows after enlarging $C$.

\subsection{Proof of Theorem~\ref{thm:conditional-2}}
Clearly the result is trivial unless $q$ is
large. For every $a\in(\Z/q\Z)^\times$, put
\[
g(a)\defeq\ssum{\ell_1,\ell_2\le q\\ \ell_1\ell_2\equiv a\bmod q}
\Bigl(1-\frac{\ell_1}{q}\Bigr)\Bigl(1-\frac{\ell_2}{q}\Bigr)
\log\ell_1\cdot\log\ell_2
=\frac1{\varphi(q)}\sum_{\chi\bmod q}\overline\chi(a)
\widetilde \vartheta (q,\chi)^2
\]
and observe that $g$ vanishes on the set
$\cE\defeq(\Z/q\Z)^\times\setminus \EOS_2(q)$. 
Then
\[
\sum_{a\in\cE}g(a)=\sum_{\chi\bmod q}\widetilde \vartheta (q,\chi)^2S(\chi)=0,
\]
where $S(\chi)\defeq\sum_{a\in\cE}\overline\chi(a)$ as before.
Since $S(\chi_0)=|\cE|$, we have
\[
\widetilde \vartheta (q,\chi_0)^2|\cE|
\le\max_{\chi\ne\chi_0}\bigl|\widetilde \vartheta (q,\chi)\bigr|
\sum_{\chi\bmod q}\bigl|\widetilde \vartheta (q,\chi)\bigr|\bigl|S(\chi)\bigr|.
\]
Using the Cauchy--Schwarz inequality
together with~\eqref{eq:E_isolate2} and~\eqref{eq:second moment weighted},
\[
\sum_{\chi\bmod q}\bigl|\widetilde \vartheta (q,\chi)\bigr|\bigl|S(\chi)\bigr|
\ll\varphi(q)(q\log q)^{1/2}|\cE|^{1/2}.
\]
Moreover, $\widetilde \vartheta (q,\chi_0)\gg q$
by~\eqref{eq:principal weighted}, and
\[
\max_{\chi\ne\chi_0}\bigl|\widetilde \vartheta (q,\chi)\bigr|
\ll q^{\sigma_0}\log q
\]
by Lemma~\ref{lem:conditional}. Putting everything together, we obtain
\[
q^2|\cE|\ll q^{\sigma_0}\log q
\cdot\varphi(q)(q\log q)^{1/2}|\cE|^{1/2},
\]
and the result  follows since $\varphi(q)\le q$.

\section*{Acknowledgments}

The authors thank the Max Planck Institute for Mathematics in
Bonn for its support and excellent working conditions during the
preparation of this paper. The second author was also
supported in part by Australian Research Council Grant
DP230100534.

The authors disclose that the large language model
Claude (Anthropic) was used in the preparation of this paper, in
particular for suggestions on writing and exposition, for literature
searches and for proof checking. The ideas behind the proofs are due
to the authors, who take full responsibility for the mathematical
content.


\begin{thebibliography}{99}

\bibitem{BS19}
W.~D. Banks and I.~E. Shparlinski,
Bounds on short character sums and $L$-functions with characters
to a powerful modulus.
\emph{J. Anal. Math.} 139 (2019), 239--263.

\bibitem{BS20}
W.~D. Banks and I.~E. Shparlinski,
Sums with the M\"obius function twisted by characters with
powerful moduli.
\emph{Trans. Amer. Math. Soc.} 373 (2020), 249--272.

\bibitem{EOS}
P.~Erd\H{o}s, A.~Odlyzko, and A.~S\'ark\"ozy,
On the residues of products of prime numbers.
\emph{Period. Math. Hungar.} 18 (1987), 229--239.

\bibitem{Iw74}
H.~Iwaniec,
On zeros of Dirichlet's $L$ series.
\emph{Invent. Math.} 23 (1974), 97--104.

\bibitem{IwKow}
H.~Iwaniec and E.~Kowalski,
\emph{Analytic number theory}.
Amer. Math. Soc. Colloq. Publ., vol.~53, Amer. Math. Soc.,
Providence, RI, 2004.

\bibitem{KMT}
O.~Klurman, A.~P. Mangerel, and J.~Ter\"av\"ainen,
Multiplicative functions in short arithmetic progressions.
\emph{Proc. Lond. Math. Soc. (3)} 127 (2023), 366--446.

\bibitem{Ko69}
N.~M. Korobov,
Double trigonometric sums and their applications to the estimation
of rational sums.
\emph{Math. Notes} 6 (1969), 472--478.

\bibitem{Ko74}
N.~M. Korobov,
The distribution of digits in periodic fractions.
\emph{Math. USSR-Sb.} 18 (1974), 659--676.

\bibitem{MT}
K.~Matom\"aki and J.~Ter\"av\"ainen,
Products of primes in arithmetic progressions.
\emph{J. Reine Angew. Math.} 808 (2024), 193--240.

\bibitem{Mo71}
H.~L. Montgomery,
\emph{Topics in multiplicative number theory}.
Lecture Notes in Math., vol.~227, Springer-Verlag, Berlin, 1971.

\bibitem{MV}
H.~L. Montgomery and R.~C. Vaughan,
\emph{Multiplicative number theory. I. Classical theory}.
Cambridge Stud. Adv. Math., vol.~97, Cambridge Univ. Press, Cambridge, 2007.

\bibitem{Po55}
A.~G. Postnikov,
On the sum of characters with respect to a modulus equal to a power
of a prime number.
\emph{Izv. Akad. Nauk SSSR. Ser. Mat.} 19 (1955), 11--16.

\bibitem{Po56}
A.~G. Postnikov,
On Dirichlet $L$-series with the character modulus equal to the
power of a prime number.
\emph{J. Indian Math. Soc. (N.S.)} 20 (1956), 217--226.

\end{thebibliography}
\end{document}